\documentclass[a4paper,10pt,reqno]{amsart}

\usepackage[UKenglish]{babel}

\usepackage{amsmath}
\usepackage{amssymb}
\usepackage{amsthm}
\usepackage{verbatim}
\usepackage{stackrel}

\usepackage[shortlabels]{enumitem}

\usepackage{comment,cite}	
\usepackage{hyperref}
\hypersetup{
    colorlinks=true,
    linkcolor=blue,
    citecolor=red,
    filecolor=magenta,      
    urlcolor=cyan,
    pdftitle={Eventual domination of semigroups and resolvents},
    bookmarks=true,
    linktocpage=true,
}

\usepackage[utf8]{inputenc}

\usepackage{amssymb}
\usepackage{amsmath}
\usepackage{amsthm}
\usepackage{bbm}
\usepackage{verbatim, stackrel}

\usepackage[shortlabels]{enumitem}
\usepackage{marginnote}
\usepackage{color}

\usepackage{stmaryrd}
\usepackage{tikz}
\usepackage{tikz-cd}

\newcommand{\bbC}{\mathbb{C}}

\newcommand{\bbN}{\mathbb{N}}

\newcommand{\bbR}{\mathbb{R}}

\newcommand{\calA}{\mathcal{A}}

\newcommand{\calL}{\mathcal{L}}

\DeclareMathOperator{\id}{id} 
\DeclareMathOperator{\one}{{\mathbbm{1}}} 
\DeclareMathOperator{\re}{Re} 
\newcommand{\argument}{\mathord{\,\cdot\,}} 
\newcommand{\dx}{\;\mathrm{d}} 
\newcommand{\norm}[1]{\left\lVert #1 \right\rVert} 
\newcommand{\modulus}[1]{\left\lvert #1 \right\rvert} 
\newcommand{\duality}[2]{\left\langle#1\, ,\, #2\right\rangle} 
\newcommand{\dom}[1]{\operatorname{dom}\left(#1\right)} 
\DeclareMathOperator{\Ima}{Rg} 
\DeclareMathOperator{\rank}{rk} 

\newcommand{\tr}{\operatorname{tr}} 
\newcommand{\nl}{\operatorname{NL}} 

\newcommand{\spec}{\sigma} 
\newcommand{\resSet}{\rho}
\newcommand{\Res}{\mathcal{R}} 
\newcommand{\spb}{s} 
\newcommand{\gbd}{\omega_0} 

\theoremstyle{definition}
\newtheorem{definition}{Definition}[section]
\newtheorem{remark}[definition]{Remark}
\newtheorem{remarks}[definition]{Remarks}
\newtheorem*{remark*}{Remark}
\newtheorem*{remarks*}{Remarks}
\newtheorem{example}[definition]{Example}

\theoremstyle{plain}
\newtheorem{proposition}[definition]{Proposition}
\newtheorem{lemma}[definition]{Lemma}
\newtheorem{theorem}[definition]{Theorem}

\numberwithin{equation}{section} 

\begin{document}

\title[Eventual domination of semigroups and resolvents]{Criteria for eventual domination of operator semigroups and resolvents}
\author{Sahiba Arora}
\address{Sahiba Arora, Technische Universität Dresden, Institut für Analysis, Fakultät für Mathematik , 01062 Dresden, Germany}
\email{sahiba.arora@mailbox.tu-dresden.de}
\author{Jochen Gl\"uck}
\address{Jochen Gl\"uck, Bergische Universität Wuppertal, Fakultät für Mathematik und Naturwissenschaften, 42119 Wuppertal, Germany}
\email{glueck@uni-wuppertal.de}
\date{\today}
\begin{abstract}
	We consider two $C_0$-semigroups $(e^{tA})_{t \ge 0}$ and $(e^{tB})_{t \ge 0}$ on function spaces (or, more generally, on Banach lattices) and analyse eventual domination between them in the sense that $\modulus{e^{tA}f} \le e^{tB}\modulus{f}$ for all sufficiently large times $t$. We characterise this behaviour and prove a number of theoretical results which complement earlier results given by Mugnolo and the second author in the special case where both semigroups are positive for large times.
	
	Moreover, we study the analogous question of whether the resolvent of $B$ eventually dominates the resolvent of $A$ close to the spectral bound of $B$. This is closely related to the so-called maximum and anti-maximum principles. In order to demonstrate how our results can be used, we include several applications to concrete differential operators.
	
	At the end of the paper, we demonstrate that eventual positivity of the resolvent of a semigroup generator is closely related to eventual positivity of the Cesàro means of the associated semigroup.
\end{abstract}

\maketitle

\section{Introduction}

\subsection*{Motivation and context}

Quite recently, the theory of positive $C_0$-semigroups on functions spaces -- or more generally, Banach lattices or ordered Banach spaces -- (see, for instance, the monographs \cite{Nagel1986, BatkaiKramarRhandi2017} and the survey paper \cite{BattyRobinson1984}) was complemented by the study of the following more subtle type of behaviour:~a $C_0$-semigroup on a Banach lattice is called \emph{eventually positive} if, for each positive initial value, the orbit becomes and stays positive for large times.
The basics of the theory are laid out in the papers \cite{DanersGlueckKennedy2016a, DanersGlueckKennedy2016b} and have later been complemented in various articles which adapt the theory, for instance, to the more subtle concept of \emph{local eventual positivity} \cite{Arora2021}.
Recent results indicate that large fractions even of the more involved parts of the theory of positive semigroups can be adapted to the eventually positive case, though significant changes to the employed methods are sometimes needed to achieve this.
We refer to \cite{Vogt2021} for an instance of a very recent spectral-theoretic result of this type.

Concrete examples of semigroups that satisfy eventual positivity or a variation thereof abound. 
To list only a few examples, we explicitly mention the semigroup generated by $-\Delta^2$ on $\bbR^d$ (see \cite{FerreroGazzolaGrunau2008, GazzolaGrunau2008}, and also \cite{FerreiraFerreira2019}), the semigroup generated by the Dirichlet-to-Neumann operator on the unit circle in $\bbR^2$ \cite{Daners2014}, certain classes of semigroups on metric graphs \cite[Proposition~5.5]{BeckerGregorioMugnolo2021}, and semigroups generated by Laplace operators which are coupled by point interactions \cite[Proposition~2]{HusseinMugnolo2020}.
This list is not exhaustive but provides an impression of the large stock of examples for eventually positive behaviour.
It is also worthwhile to note that, in finite dimensions, a general theory of eventually positive semigroups was developed earlier than its infinite-dimensional analogue; see, in particular, \cite{NoutsosTsatsomeros2008}.

Building upon the theory of eventual positivity, the closely related phenomenon of \emph{eventual domination} between two $C_0$-semigroups was studied in \cite{GlueckMugnolo2021}, where also a large variety of examples were presented.

In this article, we complement the theory of both eventual positivity and eventual domination by providing a number of natural extensions of earlier results.
In particular, we sharpen and refine some of the results on eventual domination of semigroups from \cite{GlueckMugnolo2021}, and also demonstrate that similar results can be shown for eventual domination of resolvents of operators.
Finally, we demonstrate that the interplay between eventual positivity of semigroups and of resolvents that plays an essential role in \cite{DanersGlueckKennedy2016a, DanersGlueckKennedy2016b}, can also be extended to Cesàro means of semigroups.

\subsection*{Notation and terminology}

Throughout the paper, we assume familiarity with the theory of real and complex Banach lattices for which we refer to the monographs \cite{Schaefer1974, Meyer-Nieberg1991}. So, let $E$ be a complex Banach lattice whose real part is denoted by $E_{\bbR}$. 
An operator $A:E\supseteq \dom{A} \to E$ is called \emph{real} if $\dom{A}$ is spanned by $\dom{A} \cap E_\bbR$ and $A$ maps real vectors to real vectors, i.e., it maps $\dom{A}\cap E_{\bbR}$ into $E_{\bbR}$. Note that if $A$ is a real operator, then so are the resolvent operators $\Res(\lambda,A)$ for each real number $\lambda$ in the resolvent set of $A$. In particular, an operator $T:E\to E$ is real if $TE_{\bbR}\subseteq E_{\bbR}$. Also, a $C_0$-semigroup $(e^{tA})_{t\geq 0}$ is called real if each operator $e^{tA}$ is real.

For two vectors $u,v$ in $E_\bbR$, we write $u\succeq v$ (or $v\preceq u$) if there exists a constant $c>0$ such that $u\geq cv$. With the above notation in mind, the \emph{principal ideal} of $E$ generated by a positive element $u\in E$ is given by
\[
	E_u:=\{x\in E: \modulus{x}\preceq u\}
\]
which is also a complex Banach lattice when equipped with the \emph{gauge norm} defined by
\[
	\norm{x}_u:=\inf\{c>0 : \modulus{x}\leq cu\}\quad \text{for each } x\in E_u.
\]
We say that $u$ is a \emph{quasi-interior point} of $E$ if $E_u$ is dense in $E$. For example, if $C(K)$ denotes the space of continuous functions on a compact Hausdorff space $K$, then the set of quasi-interior points is precisely the set of all $f\in C(K)$ such that $f(x)>0$ for all $x\in K$. In fact, if $u\in C(K)$ is a quasi-interior point, then $C(K)_{u}=C(K)$. On the other hand, if $(\Omega,\mu)$ is a finite measure space and 
$p\in [1,\infty)$, then the set of quasi-interior points of $L^p(\Omega,\mu)$ is the set of all functions $f\in L^p(\Omega,\mu)$ such that $f(x)>0$ for almost all $x\in \Omega$. 

For a vector $u\in E$ and a functional $\varphi\in E'$ (where $E'$ denotes the dual space of $E$), we use the notation $u\otimes \varphi$ to denote the rank-one operator
\[
	E \ni f\mapsto (u\otimes \varphi)f:= \duality{\varphi}{f}u \in E.
\]
The functional $\varphi$ is a called strictly positive if $\duality{\varphi}{f}>0$ for all $0\lneq f\in E$ or equivalently, if the kernel of $\varphi$ contains no positive non-zero element. Let $T,S: E \to E$ be linear operators. We say that $T$ is positive and write $T\geq 0$ if $Tf\geq 0$ for all $0\leq f\in E$. Of course, every positive operator is real. Moreover, for real operators $T,S: E \to E$, the notation $T\geq S$ is used to denote $T-S\geq 0$. Analogously to the case of vectors, we write $T\succeq S$ (or $S\preceq T$) if there exists $c>0$ such that $T\geq cS$.

The space of  bounded linear operators between two complex Banach spaces $E$ and $F$ is denoted as usual by $\calL(E,F)$. In addition, we use the shorthand $\calL(E):=\calL(E,E)$. The kernel of $T\in \calL(E,F)$ will be denoted by $\ker T$, the range of $T$ by $\Ima T$, and the rank of $T$ by $\rank T$. Further notations will be introduced as and when required.

\subsection*{Organisation of the article}

In Section~\ref{domination-resolvents}, we present criteria for eventual domination of resolvents -- in addition to proving sufficient criteria, we also prove a necessary condition. 
Thereafter, in Section~\ref{domination-semigroup}, we continue the theory of eventual domination of $C_0$-semigroups from \cite{GlueckMugnolo2021}. 
Applications of the results presented in Section~\ref{domination-resolvents} and~\ref{domination-semigroup} will be given in Section~\ref{applications}. 
Finally, in Section~\ref{cesaro-means}, we demonstrate that many results about the eventual positivity of semigroups and resolvents can also be adapted to Cesàro means.

\section{Eventual domination of resolvents}
	\label{domination-resolvents}

In this section, we give conditions for the resolvent of an operator $B$ to eventually dominate the resolvent of $A$ close to the spectral bound of $B$. 
This is reminiscent of the eventual domination of semigroups studied in \cite{GlueckMugnolo2021}. For generators of $C_0$-semigroups, a characterisation for domination of resolvents is given in \cite[Proposition~C-II-4.1]{Nagel1986}.
The main results of this section are the sufficient condition for eventual domination in Theorem~\ref{thm:domination-resolvents-individual} and the necessary condition in Theorem~\ref{thm:domination:converse}.

\begin{theorem}
	\label{thm:domination-resolvents-individual}
	Let $A$ and $B$ be closed, densely defined, and real operators on a complex Banach lattice $E$, let $u \in E$ be positive, and let $\lambda_0\in\bbR$ be a spectral value of $B$ and a pole of its resolvent. Moreover, assume the following:
	\begin{enumerate}[\upshape (a)]
		\item The domination conditions $\dom{A}, \dom{B} \subseteq E_u$ are satisfied.
		
		\item The spectral projection $P$ associated with the spectral value $\lambda_0$ of $B$ satisfies $Pf\succeq u$ whenever $0\lneq f\in E$.
	\end{enumerate}
	If $\lambda_0\in \rho(A)$, then for each $0\neq f\in E$, we have
	\[
		\Res(\lambda, B)\modulus{f}- \modulus{\Res(\lambda, A)f}\succeq u \quad \text{ and }\quad \Res(\mu, B)\modulus{f}+\modulus{\Res(\mu, A)f}\preceq -u
	\]
	 for all $\lambda$ in an $f$-dependent right neighbourhood of $\lambda_0$ and for all $\mu$ in an $f$-dependent left neighbourhood of $\lambda_0$.
\end{theorem}

Observe that because the operator $A$ is densely defined, the domination assumption $\dom{A}\subseteq E_u$ in particular implies that $u$ is a quasi-interior point of $E$.

The assumption~(b) in the above theorem plays a crucial role in the abstract theory of the so-called individual (anti-)maximum principles. To understand this, let $\lambda_0$ be a spectral value of  a closed and real operator $B$ on a complex Banach lattice $E$ and let $u\in E$ be a quasi-interior point. In addition, assume that $\lambda_0$ is a pole of the resolvent  and consider the inequality
\begin{equation}
	\label{eq:anti-max}
	(\lambda-\lambda_0)\Res(\lambda, B)f\geq cu
\end{equation}
for $0\lneq f\in E$, a real number $\lambda\in \rho(B)$, and some $c>0$. If there exists a constant $c>0$ such that~\eqref{eq:anti-max} holds for all $\lambda$ in an $f$-dependent \emph{right} neighbourhood of $\lambda_0$, then, in particular, we have that $\Res(\lambda,B)f\succeq u$ for all $\lambda$ in this \emph{right} neighbourhood of $\lambda_0$. In this case, $B$ is said to satisfy a \emph{individual maximum principle at $\lambda_0$}. On the other hand, if there exists a constant $c>0$ such that~\eqref{eq:anti-max} holds for all $\lambda$ in an $f$-dependent \emph{left} neighbourhood of $\lambda_0$, then, in particular, we have that $\Res(\lambda,B)f\preceq - u$ for all $\lambda$ in this \emph{left} neighbourhood of $\lambda_0$. This time, we say that $B$ satisfies an \emph{individual (anti-)maximum principle at $\lambda_0$}. For recent contributions to the abstract theory of (anti-)maximum principles, we refer to \cite{DanersGlueckKennedy2016a,DanersGlueckKennedy2016b, DanersGlueck2017, AroraGlueck2021b, AroraGlueck2022a}. 
In particular, it was proved in \cite[Theorem~4.1]{DanersGlueck2017} that if $B$ satisfies the individual maximum or the individual anti-maximum principle at $\lambda_0$, then the spectral projection $P$ corresponding to $\lambda_0$ satisfies $Pf \succeq u$ for all $0\lneq f\in E$.  If $\dom{B}\subseteq E_u$, then a converse is also true, as will be made clear below. Before moving on, we note that a characterisation of assumption~(b) above in terms of further spectral properties can be found in \cite[Corollary~3.3]{DanersGlueckKennedy2016b}.

Because of the above remarks, before proving Theorem~\ref{thm:domination-resolvents-individual}, we first show the following:

\begin{proposition}
	\label{prop:domination-resolvents-individual}
	Let $A$ and $B$ be closed, densely defined, and real operators on a complex Banach lattice $E$ and let $u \in E$ be positive. Let $\lambda_0\in\bbR$ be a spectral value of $B$ and a pole of the resolvent $\Res(\argument, B)$. Moreover, assume that the domination condition $\dom{A} \subseteq E_u$ is satisfied and that $\lambda_0\in \rho(A)$.
	\begin{enumerate}[\upshape (i)]
		\item If for each $0\lneq f\in E$, there exists $c>0$ such that~\eqref{eq:anti-max} holds
		for all $\lambda$ in an $f$-dependent \emph{right} neighbourhood of $\lambda_0$, then
		\[
			\Res(\mu, B)\modulus{f}- \modulus{\Res(\mu, A)f}\succeq u
		\]
		for each $0\neq f\in E$ and for all $\mu$ in an $f$-dependent right neighbourhood of $\lambda_0$.
		
		\item 
		If for each $0\lneq f\in E$, there exists $c>0$ such that~\eqref{eq:anti-max} holds
		 for all $\lambda$ in an $f$-dependent \emph{left} neighbourhood of $\lambda_0$, then
		\[
			\Res(\mu, B)\modulus{f}+ \modulus{\Res(\mu, A)f} \preceq -u
		\]
		 for each $0\neq f\in E$ and for all $\mu$ in an $f$-dependent left neighbourhood of $\lambda_0$.
	\end{enumerate}
\end{proposition}

The following lemma contains an observation used in the proof of Proposition~\ref{prop:domination-resolvents-individual} below. As these arguments appear frequently in the theory of\ (anti-)maximum principles, we choose to state them separately.

\begin{lemma}
	\label{lem:domination-convergence}
	Let $A$ 
	be a closed linear operator on a complex Banach lattice $E$ and let $u \ge 0$ be a quasi-interior point of $E$ such that the domination condition $\dom{A} \subseteq E_u$ is satisfied. Let $\lambda_0 \in \bbC$.
	\begin{enumerate}[\upshape (i)]
		\item If $\lambda_0$ lies in the resolvent set of $A$, then $(\lambda-\lambda_0)\Res(\lambda, A) \to 0$ in $\calL(E, E_u)$ as $\lambda \to \lambda_0$.
		
		\item If $\lambda_0$ is a simple pole of the resolvent $\Res(\argument,A)$ and the corresponding spectral projection is denoted by $P$, then $(\lambda-\lambda_0)\Res(\lambda, A) \to P$ in $\calL(E, E_u)$ as $\lambda \to \lambda_0$.
	\end{enumerate}
\end{lemma}

\begin{proof}
	We assume without loss of generality that $\lambda_0=0$.
	
	(ii) We know from \cite[Lemma~4.7(i)]{DanersGlueckKennedy2016b} that $\lambda\Res(\lambda,A)\to P$ in $\calL(E, \dom{A})$ as $\lambda\to 0$ (where $\dom{A}$ is endowed with the graph norm). Moreover, it follows from the domination assumption and the closed graph theorem that $\dom{A}$ embeds continuously into $E_u$, which implies the assertion.
	
	(i) Firstly, $0\in \rho(A)$ implies that there is a neighbourhood of $0$ which contains no spectral value of $A$ and $\Res(\lambda, A) \to \Res(0,A)$ in $\calL(E)$ as $\lambda \to 0$. Now,
	\[
		A\Res(\lambda,A)=\lambda \Res(\lambda, A) - I \to - I =A\Res(0,A)
	\]
	in $\calL(E)$ as $\lambda \to 0$. Thus $\Res(\lambda, A) \to \Res(0,A)$ in $\calL(E, \dom{A})$ as $\lambda \to 0$. The result follows as in (ii) above.
\end{proof}

\begin{proof}[Proof of Proposition~\ref{prop:domination-resolvents-individual}]
	By a change of signs, we only need to prove (i). To this end, assume without loss of generality that $\lambda_0=0$ and 
	fix a non-zero vector $f\in E$. By assumption there exist a constant $c>0$ and a number $\delta_1>0$ such that $\lambda\Res(\lambda, B)\modulus{f} \geq c u$ whenever $\lambda \in (0,\delta_1)$.
	 Furthermore, by Lemma~\ref{lem:domination-convergence}(i), there exists $\delta_2>0$ such that $\lambda \modulus{\Res(\lambda, A)f} \leq c/2 u$ for all $\lambda \in (0,\delta_2)$. 
	 Letting $\delta =\min\{\delta_1,\delta_2\}>0$, we obtain
	\[
		\lambda\big(\Res(\lambda, B)\modulus{f}- \modulus{\Res(\lambda,A)f}\big)\geq cu-\frac{c}{2} u=\frac{c}{2} u
	\]
	for all $\lambda \in (0,\delta)$, from which the result follows.
\end{proof}

To prove Theorem~\ref{thm:domination-resolvents-individual} using Proposition~\ref{prop:domination-resolvents-individual}, we will show that assumption (b) implies that~\eqref{eq:anti-max} holds in some neighbourhood of $\lambda_0$. 
Essentially, this has already been shown in \cite[Theorem~4.4]{DanersGlueckKennedy2016b}, but the independence of the constant $c$ of $\lambda$ was not explicitly mentioned there.
We state this here in a separate and general lemma, for the sake of easier reference.

\begin{lemma}
	\label{lem:convergence-strong-positivity}
	Let $u \ge 0$ be a quasi-interior point in a complex Banach lattice $E$ and let $(T_{\alpha})_{\alpha\in \calA}$ be a net of bounded and real linear operators in $\calL(E)$ whose ranges are contained in $E_u$. 
	Assume that $(T_{\alpha})_{\alpha\in \calA}$ converges strongly in $\calL(E, E_u)$ to an operator $T \in \calL(E, E_u)$.
 	
	Fix $0\lneq f\in E$. If there exists $c>0$ such that $Tf\geq cu$, then for each $\epsilon\in (0,c)$, there exists $\alpha_0\in \calA$ such that $T_{\alpha}f\geq \epsilon u$ for all $\alpha\geq \alpha_0$.
\end{lemma}

\begin{proof}
	Let $\epsilon\in (0,c)$.
	Because $(T_{\alpha})_{\alpha\in \calA}$ converges strongly to $T$ in $\calL(E, E_u)$, there exists $\alpha_0\in \calA$ such that
	\[
		\norm{T_\alpha f - Tf}_u \leq c-\epsilon
	\]
	and in turn, $\modulus{T_\alpha f - Tf} \leq (c-\epsilon)u$
	for all $\alpha\geq \alpha_0$. Since the operators $T_{\alpha}$ are all real, we can write
	\[
		T_{\alpha}f\geq Tf -(c-\epsilon) u\geq cu-(c-\epsilon) u = \epsilon u
	\]
	for all $\alpha\geq \alpha_0$.
\end{proof}

\begin{proof}[Proof of Theorem~\ref{thm:domination-resolvents-individual}]
	Without loss of generality, let $\lambda_0=0$. Using the necessary condition for our assumption (b) given in \cite[Proposition~3.1]{DanersGlueckKennedy2016b}, we get that the pole order of the spectral value $0$ of $B$ is one. 
	Thus, it follows from assumption~(a) and Lemma~\ref{lem:domination-convergence}(ii) that $\lambda \Res(\lambda, B)\to P$ in $\calL(E, E_u)$ as $\lambda \to 0$. 
		
	Fix $0\lneq f\in E$. Using assumption~(b) along with Lemma~\ref{lem:convergence-strong-positivity}, we see that~\eqref{eq:anti-max} holds for all $\lambda$ in some ($f$-dependent) neighbourhood of $\lambda_0$. The conclusion now follows from Proposition~\ref{prop:domination-resolvents-individual}(ii).
\end{proof}

We now prove a converse to Theorem~\ref{thm:domination-resolvents-individual} and Proposition~\ref{prop:domination-resolvents-individual} in the spirit of \cite[Theorem~3.1]{GlueckMugnolo2021}. 
The point of the result is that, under appropriate technical assumptions, we cannot have both eventual domination and eventual positivity of the resolvents of $A$ and $B$ on the right of an eigenvalue $\lambda_0$ of $B$ if $\lambda_0$ is also an eigenvalue of $A$ (except in the trivial case where $A = B$).

As a consequence, we are able to disprove the eventual domination of resolvents for odd order operators in Section~\ref{sec:application-odd-order}.

\begin{theorem}
	\label{thm:domination:converse}
	Let $A$ and $B$ be closed, densely defined, and real operators on a complex Banach lattice $E$ and let $u \ge 0$ be a quasi-interior point of $E$. Let $\lambda_0\in\bbR$ be a spectral value of $B$ and a pole of the resolvent $\Res(\argument, B)$. Moreover, assume the following:
	\begin{enumerate}[\upshape (a)]
		\item For each $0\leq f\in E$, we have $0\leq \Res(\lambda, A)f\leq \Res(\lambda,B)f$ for all $\lambda$ in an $f$-dependent right neighbourhood of $\lambda_0$.
		
		\item For each $0\lneq f\in E$, we have $\Res(\lambda, B)f\succeq u$ for all $\lambda$ in an $f$-dependent right neighbourhood of $\lambda_0$.
	\end{enumerate}
	If $\lambda_0$ is, in addition, a spectral value of $A$ and a pole of the resolvent $\Res(\argument,A)$, then $A=B$.
\end{theorem}

\begin{proof}
	There is no loss of generality in assuming $\lambda_0=0$. Denote by $P$ the spectral projection associated with the spectral value $0$ of $B$. 
	The eventual positivity assumption in (b) and \cite[Theorem~4.1]{DanersGlueck2017} together imply that $Pf\succeq u$ whenever $0\lneq f\in E$. We can thus apply \cite[Corollary~3.3]{DanersGlueckKennedy2016b} to obtain that $\ker B$ is spanned by a vector $v\succeq u $, the eigenspace $\ker B'$ contains a strictly positive functional $\psi$, and $0$ is also a simple pole of the resolvent $\Res(\argument, B)$. The latter implies that $\Ima P=\ker B$ and $\Ima P'=\ker B'$. 
	We may rescale $\psi$ to have norm $1$. 
	Furthermore, we rescale $v$ such that $\duality{\psi}{v}=1$, and thus obtain  $P=v\otimes \psi$.
	
	Next, let $Q$ denote the spectral projection of $A$ associated with the pole $0$. 
	As $0$ is a simple pole of $\Res(\argument,B)$, we conclude from the eventual domination assumption in~(a) that $0$ is also a simple pole of $\Res(\argument,A)$.
	So, $\lambda \Res(\lambda,A)\to Q$ in $\calL(E)$ as $\lambda \to 0$. 
	Moreover, the eventual positivity assumption on $\Res(\argument,A)$ in~(a) implies that $Q\geq 0$ and the eventual domination in assumption (a) implies that $P\geq Q\geq 0$. We can thus infer from \cite[Proposition~2.1.3]{Emelyanov2007} that $1 = \rank P \ge \rank Q > 0$, so $\rank Q = 1$.
	Thus there exists $w\in W$ and $\phi \in E'$ such that $Q=w\otimes \phi$; here, we may choose $\phi$ to have norm $1$. 
	Moreover, due to the positivity of $Q$, we may choose both $w$ and $\phi$ to be positive. 
	Note that $\duality{\phi}{w}=1$, because $Q$ is a projection.
	
	Next, we prove that $\psi=\phi$. To this end, observe that $\phi=Q'\phi\leq P'\phi= \duality{\phi}{v}\psi$. Since
	\[
		\duality{\duality{\phi}{v}\psi-\phi}{v}=0
	\]
	and $v$ is a quasi-interior point of $E$, it follows that $\duality{\varphi}{v}\psi=\phi$; see \cite[Theorem~II.6.3]{Schaefer1974}. 
	As $\psi$ and $\phi$ are both positive functionals of norm $1$, this implies that $\psi=\phi$.
	
	Now consider a vector $0\lneq f\in E$ and let $\delta>0$ be such that $\Res(\lambda,B)f \geq \Res(\lambda,A)f$ for all $\lambda \in (0,\delta)$. Then, for $\lambda \in (0,\delta)$, the computation
	\begin{align*}
		\duality{\psi}{\Res(\lambda,B)f-\Res(\lambda,A)f}
		&= 
		\duality{\Res(\lambda,B')\psi}{f}-\duality{\Res(\lambda,A')\phi}{f} \\
		&= 
		\frac1{\lambda} [\duality{\psi}{f}-\duality{\phi}{f}]
		=
		0
	\end{align*}
	and the strict positivity of $\psi$ imply that $\Res(\lambda,B)f=\Res(\lambda,A)f$ for all $\lambda \in (0,\delta)$. 
	Since the resolvent is analytic, the identity theorem for analytic functions yields that the same equality holds for all $\lambda$ in a non-empty $f$-independent set $U \subseteq \bbC$.
	Hence, $\Res(\argument,A)=\Res(\argument,B)$ for all $\lambda \in U$, and we thus conclude that $A=B$, as desired.
\end{proof}

\section{Eventual domination of semigroups}
	\label{domination-semigroup}

We now consider the case of $C_0$-semigroups rather than resolvents.
Eventual domination for $C_0$-semigroups was studied in \cite{GlueckMugnolo2021}; 
in this section, we provide a number of refinements and improvements to the results given there.
Our first result, Theorem~\ref{thm:domination-semigroups-individual}, extracts sufficient conditions for eventual domination from \cite[Theorem~3.1]{GlueckMugnolo2021}, optimises the assumptions of the aforementioned theorem, and further improves the conclusion.
In particular, in contrast to \cite[Theorem~3.1]{GlueckMugnolo2021}, we do not need the dominated semigroup to be eventually positive, and we also consider non-positive initial values $f$ (the latter is made possible by the inclusion of the modulus in the eventual domination estimate).

To state the theorem, we recall that for a closed operator $A$ on a Banach space $E$, its \emph{spectral bound} is defined as
\[
	\spb(A):= \sup\{\re \lambda: \lambda \in\spec(A)\}\in [-\infty,\infty].
\]
If $A$ generates a $C_0$-semigroup, then $\spb(A) < \infty$.

For two $C_0$-semigroups $(e^{tA})_{t\geq 0}$ and $(e^{tB})_{t\geq 0}$ on a complex Banach lattice $E$ we say that $(e^{tA})_{t\geq 0}$ \emph{individually eventually dominates} $(e^{tB})_{t\geq 0}$ if for each $f \in E$ there exists a time $t_0 \ge 0$ such that
\[
	\modulus{e^{tA}f} \le e^{tB} \modulus{f}
\]
for all $t \ge t_0$.

\begin{theorem}
	\label{thm:domination-semigroups-individual}
	Let $(e^{tA})_{t\geq 0}$ and $(e^{tB})_{t\geq 0}$ be real $C_0$-semigroups on a complex Banach lattice $E$ and let $u \ge 0$ be a quasi-interior point of $E$. In addition, assume the following: 
	\begin{enumerate}[\upshape (a)]
		\item 
		There exist  $t_1,t_2 \geq 0$ such that $e^{t_1 A}E, e^{t_2 B}E\subseteq E_u$.

		\item 
		The re-scaled semigroup $(e^{t(B-\spb(B))})_{t\geq 0}$ converges strongly as $t\to\infty$ to an operator $P$ that satisfies $Pf\succeq u$ whenever $0\lneq f\in E$.
		
		\item 
		The rescaled semigroup $(e^{t(A-\spb(B))})_{t\geq 0}$ converges strongly to $0$ as $t\to\infty$.
	\end{enumerate}
	Then for each non-zero $f\in E$, there exists a time $\tau\geq 0$ such that $e^{tB}\modulus{f}-\modulus{e^{tA}f}\succeq u$ for all $t\geq \tau$. In particular, $(e^{tB})_{t\geq 0}$ individually eventually dominates $(e^{tA})_{t\geq 0}$.
\end{theorem}

The assumption~(b) in the previous theorem plays an important role in the theory of individual eventual positivity. Let $(e^{tB})_{t\geq 0}$ be a real $C_0$-semigroup on a complex Banach lattice $E$ and let $0 \le u\in E$ be a quasi-interior point. If for each $0\lneq f\in E$, there exists a time $t_0\geq 0$ such that $e^{tB}f\succeq u$ for all $t \ge t_0$, then the semigroup $(e^{tB})_{t\geq 0}$ is said to be \emph{individually eventually strongly positive} (with respect to $u$). The theory of eventually positive semigroups was developed in several recent articles. In particular, it was shown in \cite[Theorem~5.1]{DanersGlueck2017} that if the spectral bound $\spb(B)$ is a spectral value of $B$ and a pole of the resolvent $\Res(\argument,B)$, then individual eventual strong positivity of $(e^{tB})_{t\geq 0}$ with respect to $u$ implies that the spectral projection $P$ associated with $\spb(B)$ satisfies $Pf\succeq u$ for all $0\lneq f\in E$. 

We prove the following result before proving Theorem~\ref{thm:domination-semigroups-individual}:

\begin{proposition}
	\label{prop:domination-semigroups-individual}
	Let $(e^{tA})_{t\geq 0}$ and $(e^{tB})_{t\geq 0}$ be real $C_0$-semigroups on a complex Banach lattice $E$ and let $u \ge 0$ be a quasi-interior point of $E$. In addition, assume the following: 
	\begin{enumerate}[\upshape (a)]
		\item There exists  $t_1 \geq 0$ such that $e^{t_1 A}E\subseteq E_u$.
		
		\item For each $0\lneq f\in E$, there exists a time $t_0\geq 0$ and a constant $c>0$ such that $e^{tB}f \geq cu$ for all $t\geq t_0$.
		
		\item The semigroup $(e^{tA})_{t\geq 0}$ converges to $0$ strongly as $t\to\infty$.
	\end{enumerate}
	Then for each non-zero $f\in E$, there exists a time $\tau\geq 0$ such that $e^{tB}\modulus{f}-\modulus{e^{tA}f}\succeq u$ for all $t\geq \tau$. In particular, $(e^{tB})_{t\geq 0}$ individually eventually dominates $(e^{tA})_{t\geq 0}$.
\end{proposition}

\begin{proof}
	First of all, note that by the closed graph theorem we have $e^{t_1 A}\in \calL(E, E_u)$. Therefore $e^{(t+t_1)A} = e^{t_1 A} e^{tA} \to 0$ 
	 in $\calL(E, E_u)$ as $t\to \infty$.
	Fix a non-zero vector $f\in E$. Then there exists $t_2\geq 0$ such that
	\[
		\norm{e^{(t+t_1)A}f}_u \leq \frac{c}{2}
	\]
	for all $t\geq t_2$, where $c$ is the constant from assumption (b) for the vector $\modulus{f}$. 
	Hence $\modulus{e^{tA}f}\leq \frac{c}{2}u$ for all $t\geq t_1+t_2$. Let $\tau=\max\{t_0,t_1+t_2\}$. Then
	\[
		e^{tB}\modulus{f}-\modulus{e^{tA}f}\geq cu-\frac{c}{2}u=\frac{c}{2}u
	\]
	for all $t\geq \tau$, as required.
\end{proof}

To prove Theorem~\ref{thm:domination-semigroups-individual}, we need to know that the strong convergence $e^{t(B-\spb(B))}\to P$ as $t\to \infty$ implies the assumption (b) in Proposition~\ref{prop:domination-semigroups-individual} for the semigroup generated by $B-\spb(B)$. This was essentially shown in \cite[Theorem~5.2]{DanersGlueckKennedy2016b} but the independence of the constant $c$ of time $t$ was not mentioned explicitly there; it does, however, follow from Lemma~\ref{lem:convergence-strong-positivity}.
The details are as follows.

\begin{proof}[Proof of Theorem~\ref{thm:domination-semigroups-individual}]
	Replacing $B$ by $B-\spb(B)$ and $A$ by $A-\spb(B)$, we may assume that $\spb(B)=0$.
	We only need to prove that assumption (b) in Proposition~\ref{prop:domination-semigroups-individual} is satisfied. To this end, fix $0\lneq f\in E$ and let $c>0$ be such that $Pf\geq cu$. 
	Since $e^{t_2 B}E\subseteq E_u$, the operator $e^{t_2 B}$ is bounded from $E$ to $E_u$ by the closed graph theorem. We therefore obtain that
	\[
		e^{tB}= e^{t_2 B} e^{(t-t_2)B} \to P \text{ in } \calL(E,E_u)
	\]
	as $t\to\infty$. 
	Combining this with assumption~(b) and Lemma~\ref{lem:convergence-strong-positivity}, we obtain that for each $0\lneq f\in E$, there exist a constant $c>0$ and time $t_0\geq 0$ such that $e^{tB}f\geq cu$ for all $t\geq t_0$.
Hence, all assumptions of Proposition~\ref{prop:domination-semigroups-individual} are satisfied, so the result follows.
\end{proof}

After obtaining sufficient conditions for individual domination of semigroups, we now turn to the uniform case. 
For $C_0$-semigroups $(e^{tA})_{t\geq 0}$ and $(e^{tB})_{t\geq 0}$ on a complex Banach lattice $E$ we say that $(e^{tA})_{t\geq 0}$ \emph{uniformly eventually dominates} $(e^{tB})_{t\geq 0}$ if there exists a time $t_0 \ge 0$ such that
\[
	\modulus{e^{tA}f} \le e^{tB} \modulus{f}
\]
for all $t \ge t_0$ and all $f \in E$.
Before giving a sufficient condition for uniform eventual domination, we show by means of an example that the notions of individual and uniform eventual domination are not equivalent.

\begin{example}
	Consider the Banach lattice $E=C[0,1]$ and the strictly positive functionals $\phi_A,\phi_B: E\to \mathbb C$ given by
	\[
		\duality{\phi_A}{f} = \int_0^1 f(x)\dx x\quad \text{ and }\quad \duality{\phi_B}{f} = 2\int_0^1 xf(x)\dx x
	\]
	for all $f\in E$. Let $P_A$ and $P_B$ be the rank-one projections on $E$ given by $P_A=\one\otimes \phi_A$ and $P_B=\one\otimes \phi_B$; where $\one\in E$ denotes the constant one function.
	
	Now consider the bounded linear operators
	\[
		A=P_A-\frac32 \id\quad \text{ and }\quad B=P_B-\id
	\]
	on $E$. Then $\spb(B)=0$ and $e^{tB}=P_B+e^{-t}(\id-P_B) \to P_B$ in operator norm as $t\to \infty$. 
	Moreover, $\spb(A) = -\frac12$, which implies -- since $A$ is a bounded operator -- that $(e^{tA})_{t\geq 0}$ converges in operator norm to $0$ as $t\to\infty$. Furthermore, because $E_{\one}=E$ and $P_Bf\succeq \one$ whenever $0\lneq f\in E$, all the assumptions of Theorem~\ref{thm:domination-semigroups-individual} are satisfied. 
	Hence, $(e^{tB})_{t\geq 0}$ individually eventually dominates $(e^{tA})_{t\geq 0}$. 
	
	Next, note that $P_B$ is the spectral projection of $B$ associated with $\spb(B)=0$. As $\spb(A) < 0$, Theorem~\ref{thm:domination-resolvents-individual} implies that for each $f\geq 0$, we have $\Res(\lambda, B)f- \Res(\lambda, A)f\succeq \one$ for all $\lambda$ in an $f$-dependent right neighbourhood of $0$.
	
	In order to see that the eventual domination is not uniform, we define, for each $n\in \bbN$, a function $0 \le f_n \in E$ in the following way:~let $f_n(x)= 1-nx$ for $x \in [0,1/n)$ and $f_n(x) = 0$ for $x \in [1/n,1]$. 
	Then
	\[
		P_Af_n = \frac{1}{2n}\one \qquad \text{and} \qquad P_Bf_n=\frac{1}{3n^2}\one.
	\]
	Using $f_n(1)=0$, we obtain
	\begin{align*}
		(e^{tB} - e^{tA})f_n(1) &= \frac{1}{3n^2}(1-e^{-t}) -\frac{e^{-t/2}}{2n}(1-e^{-t})\\
						&= \left(\frac{1}{3n^2} -\frac{e^{-t/2}}{2n}\right)(1-e^{-t})
	\end{align*}
	for all $t\geq 0$. 
	Thus, for each fixed $t> 0$, we conclude for all sufficiently large $n$ (namely, for $n>2/3 e^{t/2}$) that $(e^{tB} - e^{tA})f_n \not\geq 0$. 
	Therefore, $(e^{tB})_{t\geq 0}$ does not uniformly eventually dominate $(e^{tA})_{t\geq 0}$.
	
	Similarly, using the simple fact that the resolvent of every projection $P$ is given by
	\begin{align*}
		\Res(\lambda,P) = \frac{P + (\lambda - 1)\id}{\lambda(\lambda-1)}
	\end{align*}
	for all $\lambda \in \bbC \setminus \{0,1\}$, we can compute that 
	\[
		(\Res(\lambda,B)-\Res(\lambda,A))f_n(1) 
		= 
		\frac{1}{3n^2}\left( \frac{1}{\lambda(\lambda+1)}\right) - \frac2n\left(\frac{1}{(2\lambda+3)(2\lambda+1)}\right)
	\]
	for each $\lambda>0$. Hence, for each fixed $\lambda > 0$ we have
	\[
		\Res(\lambda,B)f_n(1) < \Res(\lambda,A)f_n(1)
	\]
	for all sufficiently large $n$.
\end{example}

Conditions for uniform eventual domination between semigroups were given in \cite[Theorem~3.7]{GlueckMugnolo2021}. 
However, only the case of self-adjoint semigroups on $L^2$-spaces was considered there.
By adapting techniques from \cite{DanersGlueck2018b} we now give a sufficient condition for uniform eventual domination which also works for semigroups on more general spaces (and thus, in particular, without any self-adjointness assumption).
As in Theorem~\ref{thm:domination-semigroups-individual}, we do not need the dominated
 semigroup to be (eventually) positive, and by including a modulus on both sides of the eventual domination estimate, we are able to consider non-positive initial values as well.

\begin{theorem}
	\label{thm:domination-semigroups-uniform}
	Let $(e^{tA})_{t\geq 0}$ and $(e^{tB})_{t\geq 0}$ be real $C_0$-semigroups on a complex Banach lattice $E$. Let $u \ge 0$ be a quasi-interior point of $E$ and $\phi\in E'$ be a strictly positive functional. In addition, assume the following: 
	\begin{enumerate}[\upshape (a)]
		\item There exist times $t_1, t_2 \geq 0$ such that $e^{t_1 A}E\subseteq E_u$ and $e^{t_2 A'}E'\subseteq (E')_{\varphi}$.
		
		\item There exists a time $t_3\geq 0$ and a constant $c>0$ such that $e^{tB} \geq  c(u\otimes \varphi)$ for all $t\geq t_3$.
		
		\item The semigroup $(e^{tA})_{t\geq 0}$ converges to $0$ in operator norm as $t\to\infty$.
	\end{enumerate}
	Then there exists a time $\tau\geq 0$ such that $e^{tB}\modulus{f}- \modulus{e^{tA}f} \succeq (u \otimes \varphi) \modulus{f}$ for all $t\geq \tau$ and all $f\in E$. In particular, $(e^{tB})_{t\geq 0}$ uniformly eventually dominates $(e^{tA})_{t\geq 0}$.
\end{theorem}

We make a few comments on the assumptions of the above theorem before proving the result.

\begin{remarks}
	\label{remark:domination-semigroups-uniform}
	(a) 
	\emph{Smoothing assumptions}. Assumption (a) in Theorem~\ref{thm:domination-semigroups-uniform} can be seen as an abstract 
	version of an ultracontracitivity property. The idea to use this assumption to obtain uniform eventual domination is directly borrowed from \cite[Theorem~3.1]{DanersGlueck2018b}, where similar smoothing assumptions were imposed to obtain uniform eventual positivity of semigroups.
	
	(b) 
	\emph{Eventual positivity assumption}. Assumption (b) in above is a uniform eventual strong positivity assumption and is in line with the individual eventual strong positivity assumption in Proposition~\ref{prop:domination-semigroups-individual} on the dominating semigroup. In fact, our assumption here is exactly the conclusion of \cite[Theorem~3.1]{DanersGlueck2018b} (in the case where $\spb(B)=0$), and is therefore satisfied whenever the hypothesis of \cite[Theorem~3.1]{DanersGlueck2018b} is. In particular, assumption (b) can be replaced by the following two conditions:
	\begin{enumerate}[\upshape (i)]
		\item There exist times $t_3, t_4\geq 0$ such that $e^{t_3 B}E\subseteq E_u$ and $e^{t_4 B'}E'\subseteq (E')_{\varphi}$.
	
		\item The spectral bound $\spb(B)$ is a dominant spectral value of $B$ and the corresponding eigenspace is generated by a vector $v\succeq u $. Moreover, the dual eigenspace $\ker(\spb(B)-B')$ contains a functional $\psi\succeq \varphi$.
	\end{enumerate}
	In this framework, it is worth mentioning an inaccuracy in the statement of the conclusion of \cite[Theorem~3.1]{DanersGlueck2018b}: the estimate $e^{tA} \ge \epsilon (u \otimes \varphi)$ for all $t \ge t_0$ in the theorem is clearly incorrect in cases where the semigroup converges to $0$ as $t \to \infty$. 
	This is due to the lack of a scaling factor: as is obvious from the proof, the correct conclusion there is $e^{t(A-\spb(A))} \ge \epsilon (u \otimes \varphi)$ for all $t \ge t_0$.
	
	(c) 
	\emph{Convergence in the operator norm}. 
	In order to obtain individual eventual domination in Proposition~\ref{prop:domination-semigroups-individual}, we assumed there that the dominated semigroup converges strongly to $0$ as $t\to\infty$. Combining this with the assumption $e^{t_1A}E\subseteq E_u$ for some $t_1\geq 0$, we were able show that $e^{tA}f\preceq u$ for sufficiently large $t$, which was an important ingredient for the proof of individual eventual domination. 
	In the same vein, in order to obtain uniform eventual domination, we require the dominated semigroup to converge to $0$ in the operator norm, and we will combine this with the smoothing assumption (a) in Theorem~\ref{thm:domination-semigroups-uniform} to obtain a similar estimate.
	
	We note that operator norm convergence of $C_0$-semigroups has been investigated in detail in \cite{Thieme1998}.
	
	(d)
	By a simple rescaling argument we can clearly replace the estimate in assumption~(b) by the estimate $e^{t(B-s(B))} \geq  c(u\otimes \varphi)$; in this case we also have to replace assumption~(c) by the assumption that $(e^{t(A - s(B))})_{t\geq 0}$ converges to $0$ in operator norm.
\end{remarks}

In the following proof not only do we need the principal ideal $E_u$, but also the Banach lattice $E^\varphi$ which is, for a strictly positive functional $\varphi \in E'$, defined to be the completion of $E$ with respect to the norm $\duality{\varphi}{\modulus{\argument}}$. This space plays an essential role in \cite{DanersGlueck2018b} and \cite{AroraGlueck2021b}; we also refer to these references for more details about $E^\varphi$.

\begin{proof}[Proof of Theorem~\ref{thm:domination-semigroups-uniform}]
	We proceed as in the proof of \cite[Theorem~3.1]{DanersGlueck2018b}. 
	For each $t \ge 0$ we deduce from assumption~(a) together with \cite[Proposition~2.5]{AroraGlueck2021b} that
	\[
		e^{(t+t_1+t_2)A} = e^{t_1 A} e^{tA} e^{t_2 A}
	\]
	extends to a bounded linear operator from $E^\varphi$ to $E_u$ for each time $t\geq 0$. Furthermore,
	\[
		\norm{e^{(t+t_1+t_2)A}}_{E^\varphi \to E_u} \leq \norm{e^{t_1A}}_{E\to E_u} \norm{e^{tA}}_{E\to E} \norm{e^{t_2 A}}_{E^\varphi \to E},
	\]
	and the latter converges to $0$ as $t\to \infty$. 
	This gives the existence of $t_0\geq 0$ such that $\norm{e^{(t+t_1+t_2)A}}_{E^\varphi \to E_u} \leq c/2$ for all $t\geq t_0$; here the constant $c>0$ is the same one as in assumption~(b).
	
	Fix $f\in E$. Then the preceding inequality yields
	\[
		\norm{e^{tA}f}_u \leq \frac{c}{2} \norm{f}_{E^\varphi} = \frac{c}{2} \duality{\varphi}{\modulus{f}} 
	\]
	for all $t\geq t_0+t_1+t_2$. 
	We thus obtain $\modulus{e^{tA}f} \leq \frac{c}{2} \duality{\varphi}{\modulus{f}} u= \frac{c}{2} (u\otimes \varphi) \modulus{f}$ for all $t\geq t_0+t_1+t_2$. Let $\tau :=\max\{t_0+t_1+t_2, t_3\}$. The eventual positivity assumption~(b) on $(e^{tB})_{t\geq 0}$ along with the above inequality gives
	\[
		e^{tB}\modulus{f}-\modulus{e^{tA}f}\geq c (u\otimes \varphi)\modulus{f} - \frac{c}{2} (u\otimes \varphi)\modulus{f}=\frac{c}{2} (u\otimes \varphi)\modulus{f}
	\]
	for all $t\geq \tau$. 
	This proves the assertion.
\end{proof}

\section{Applications}
	\label{applications}

In this section, we demonstrate how our abstract results can be applied to a variety of differential operators. Several applications of eventual domination of semigroups were already given in \cite[Section~4]{GlueckMugnolo2021}, but now we have more freedom as our dominated semigroup is not required to be eventually positive. Moreover, the theory in Section~\ref{domination-resolvents} enables us to prove eventual domination of the resolvents as well.
Our intention is mainly to illustrate that eventual domination of semigroups and resolvents can happen in various situations, so we try to avoid technical difficulties by keeping the differential operators in our examples rather simple.

\subsection{The Laplace operator with anti-symmetric boundary conditions}

Let us consider the realisation of the Laplace operator on $L^2(-1,1)$ with anti-symmetric boundary conditions given by
\begin{align*}
	\dom{\Delta^{AS}} &= \{f \in H^2(-1,1): f(-1) = -f(1), \; f'(-1) = -f'(1)\}, \\ 
	\Delta^{AS} f     &= f''.
\end{align*}

This is a self-adjoint operator on $L^2(-1,1)$ and $-\Delta^{AS}$ is associated with the sesqui-linear form $a$ given by $a(u,v) = \int_{-1}^1 u \overline{v} \dx x$ with form domain
\begin{align*}
	\dom{a} = \{f \in H^1(-1,1): \; f(-1) = -f(1)\}.
\end{align*}
Since $H^2(-1,1)$ embeds compactly into $L^2(-1,1)$, it follows that $\Delta^{AS}$ has compact resolvent and hence all spectral values of $\Delta^{AS}$ are eigenvalues.
Moreover, a straightforward computation shows that the eigenvalues of $\Delta^{AS}$ are precisely the numbers
\begin{align*}
	-\left(k+\frac{1}{2}\right)^2 \pi^2 \quad \text{for integers } k \ge 0,
\end{align*}
each of them with multiplicity $2$; 
the eigenspace of the leading eigenvalue $\spb(\Delta^{AS}) = -\pi^2/4$ is spanned by the functions $u_1, u_2$ given by
\begin{align*}
	u_1(x) = \cos\left(\frac{\pi}{2}x\right) 
	\qquad \text{and} \qquad 
	u_2(x) = \sin\left(\frac{\pi}{2}x\right)
\end{align*}
for $x \in (-1,1)$.
Since $\Delta^{AS}$ is self-adjoint, the spectral projection $P$ associated with $\spb(\Delta^{AS})$ is the orthogonal projection onto the eigenspace, given by
\begin{align*}
	P = c \big( u_1 \otimes u_1 + u_2\otimes u_2 \big)
\end{align*}
for an appropriate normalisation constant $c > 0$.
In particular, the operator $P$ is not positive since, for instance, $\left(P \one_{(0,1)}\right)(x)$ is strictly negative for $x$ close to $-1$.

Due to the non-positivity of the projection $P$, the semigroup $(e^{t \Delta^{AS}})_{t \ge 0}$ is not positive and, more generally, not even eventually positive \cite[Theorem~8.3]{DanersGlueck2018b} (in fact, this reference shows that the semigroup does not even have the weaker property of being \emph{asymptotically positive}).

Now, as a second operator, we consider the Neumann Laplacian $\Delta^{N}$ on $L^2(-1,1)$. 
It generates a positive semigroup, has spectral bound $0$, and the corresponding eigenspace is spanned by the constant function $\one$.
We have the following comparison result.

\begin{theorem}
	There exists a time $\tau \ge 0$ such that the estimate
	\begin{align*}
		e^{t\Delta^N} \modulus{f} - \modulus{e^{t\Delta^{AS}}f}\succeq (\one \otimes \one) \modulus{f}
	\end{align*}
	holds for all $t \ge \tau$ and all $f \in L^2(-1,1)$.
	On the other hand, the semigroup $(e^{t\Delta^N})_{t \ge 0}$ does not dominate $(e^{t\Delta^{AS}})_{t \ge 0}$ for all times, i.e., there exists a time $t > 0$ and a function $f \in L^2(-1,1)$ such that
	\begin{align*}
		e^{t\Delta^N} \modulus{f} \not\geq \modulus{e^{t\Delta^{AS}}f}.
	\end{align*}
\end{theorem}
\begin{proof}
	The eventual domination claim follows from Theorem~\ref{thm:domination-semigroups-uniform}, applied to the vectors $u = \varphi = \one$. 
	(There are various ways to see that assumption~(b) in the theorem is satisfied for $B = \Delta^N$; within the framework of eventually positive semigroups, one can, for instance, use \cite[Theorem~3.1]{DanersGlueck2018b} to see this.)
	
	The fact that one does not have domination for all times can be easily checked by employing form methods: 
	one can simply use the characterisation of domination given in \cite[Theorem~2.21]{Ouhabaz2005}, since the form domain $\dom{a}$ introduced above is not an ideal (in the sense of \cite[Definition~2.19]{Ouhabaz2005}) of the form domain $H^1(-1,1)$ of the sesqui-linear form associated with $-\Delta^N$.
\end{proof}

\begin{remark}
	With essentially the same argument, one could show that the Laplace
	operator with periodic boundary conditions eventually dominates the
	Laplace operator with anti-symmetric boundary conditions.
	However, in this case, the domination even holds for all times, as was
	shown in \cite[Example~2]{Ouhabaz1996}.
\end{remark}

\subsection{Laplace operators with non-local boundary conditions}

Our next application concerns an example that is used in several articles about eventually positive semigroups and (anti-)maximum principles -- the Laplacian on an interval with certain non-local boundary conditions. Eventual positivity of the generated semigroup can be nicely demonstrated via two specific choices of non-local boundary conditions -- one in which the boundary conditions are symmetric and the other subject to non-symmetric boundary conditions; see, for instance,\cite[Section~11.7]{Glueck2016}. Here, we deal with the latter. 
We prove the following result.

\begin{theorem}
	\label{thm:non-local-semigroup-dominance}
	Let $\beta_1,\beta_2\in \bbR$. On $E:=L^2(0,\pi)$, we consider the Laplace operators $\Delta_i$ with domains
	\[
		\dom{\Delta_i} =\{f\in H^2(0,\pi) : f'(0)=0, f'(\pi)=\beta_i f(0)\}
	\]
	for $i=1,2$.
	Let $\one\in E$ denote the constant function taking the value one and 
	assume that $-\frac12<\beta_1<\beta_2<0$. 
	Then there exists a time $\tau\geq 0$ such that
	\[
		 e^{t\Delta_2}\modulus{f} - \modulus{e^{t\Delta_1}f}  \succeq  (\one \otimes \one) \modulus{f} 
	\]
	for all $t\geq \tau$ and all non-zero $f\in E$.
		
	Moreover, for each $0\neq f\in E$, we have
	\[
		\Res(\lambda, \Delta_2)\modulus{f}- \modulus{\Res(\lambda, \Delta_1)f}\succeq \one \quad \text{ and }\quad \Res(\mu, \Delta_2)\modulus{f}+\modulus{\Res(\mu, \Delta_1)f}\preceq -\one
	\]
	 for all $\lambda$ in an $f$-dependent right neighbourhood of $\spb(\Delta_2)$ and for all $\mu$ in an $f$-dependent left neighbourhood of $\spb(\Delta_2)$.	
\end{theorem}

For a fixed $\beta\in\bbR$, the operator $\Delta_{\beta}$ introduced above was considered in \cite[Theorem~11.7.4]{Glueck2016}, where necessary and sufficient conditions were given for the corresponding semigroup to be positive and individually eventually positive. It can actually be shown using \cite[Theorem~3.1]{DanersGlueck2018b} that the eventual positivity assertion there is uniform. In fact, with methods similar to \cite[Proposition~6.3]{AroraGlueck2021b}, one can even prove uniform (anti-)maximum principles at the spectral bound.
We mention that the operator considered here is a slight modification of the simple thermostat models considered in \cite{GuidottiMerino2000,GuidottiMerino1997}.

\begin{proof}[Proof of Theorem~\ref{thm:non-local-semigroup-dominance}]
	Firstly, the semigroup is analytic and
	\[
		\dom{\Delta_1} \subseteq H^1(0,\pi)\subseteq L^\infty(0,\pi) = E_{\one},
	\]
	so, 
	\[
		e^{t\Delta_1} E\subseteq \cap_{n\in\bbN} \dom{\Delta_1^n} \subseteq E_u
	\]
	for all $t>0$. Also, because $\Delta_1$ is real, the Hilbert space adjoint $\Delta_1^*$ coincides with the Banach space dual $\Delta_1'$ (one can find this argument in, for example, \cite[page~46]{DanersGlueck2018b}). Moreover, the structure of $\Delta_1^*$ is same as that of $\Delta_1$ with switched end points of the interval. It follows that assumption~(a) of Theorem~\ref{thm:domination-semigroups-uniform} is satisfied for $A:=\Delta_1-\spb(\Delta_2)$. Moreover, assumption~(b) of Theorem~\ref{thm:domination-semigroups-uniform} is also satisfied with $B:=\Delta_2-\spb(\Delta_2)$ (see Remark~\ref{remark:domination-semigroups-uniform}(b)) and $u=\varphi=\one$. 
	Indeed, the proof is exactly the same as that of \cite[Theorem~4.3]{DanersGlueck2018b} (the only difference is that one needs to use \cite[Theorem~11.7.4]{Glueck2016} instead of \cite[Theorem~6.10]{DanersGlueckKennedy2016b}).
	
	Since $A$ generates an analytic semigroup, the growth bound of the semigroup coincides with $\spb(A)$.
	So, if we are able to show that $\spb(A)<0$, then assumption~(c) of Theorem~\ref{thm:domination-semigroups-uniform} also follows. 
	To see that indeed $\spb(A) < 0$, we make use of the computations performed in \cite[Theorem~11.7.4(c)]{Glueck2016} -- which not only show that both $\spb(\Delta_1)$ and $\spb(\Delta_2)$ are negative, but also that the numbers $\mu_i=\sqrt{-\spb(\Delta_i)}>0$ (for $i=1,2$) are the only numbers in $(0, 1/2)$ that satisfy the equations
	\[
		\mu_1\sin(\mu_1\pi)=-\beta_1 \qquad \text{and}\qquad\mu_2\sin(\mu_2\pi)=-\beta_2,
	\]
	respectively. Since the function
	\[
		\mu\mapsto\mu\sin(\mu \pi)
	\]
	is strictly increasing on $(0,\frac12)$ and $\beta_1<\beta_2$, it follows that $\mu_1>\mu_2$ and so $\spb(\Delta_1)<\spb(\Delta_2)$. Hence, we have shown that $\spb(A)<0$.
	
	Thus, we infer from Theorem~\ref{thm:domination-semigroups-uniform}, the existence of a time $\tau>0$ such that
	\[
		e^{t(\Delta_2-\spb(\Delta_2))}\modulus{f}- \modulus{e^{t(\Delta_1-\spb(\Delta_2))}f} \succeq (\one \otimes \one) \modulus{f}
	\]
	for all $t\geq \tau$ and all non-zero $f\in E$, which proves the first conclusion.
	
	Lastly, the inequality $\spb(\Delta_1)<\spb(\Delta_2)$ above also implies that $\spb(\Delta_2)\in \rho(\Delta_1)$. 
	Moreover, the fact that the spectral projection $P$ of $\Delta_2$ associated with $\spb(\Delta_2)$ satisfies $Pf\succeq \one$ for all $0\lneq f\in E$ follows from \cite[Theorem~11.7.4]{Glueck2016} (use \cite[Corollary~3.3]{DanersGlueckKennedy2016b}). 
	The second assertion is thus a consequence of Theorem~\ref{thm:domination-resolvents-individual}.
\end{proof}

\subsection{Dirichlet and Neumann boundary conditions -- what is (eventually) in between?}

The question of which semigroups are sandwiched between the Laplace operators with Dirichlet and Neumann boundary conditions dates back to Arendt and Warma \cite{ArendtWarma2003}. Under locality and regularity assumptions, it was shown that the sandwiched semigroup is generated by a Laplace operator with general Robin boundary conditions. Later, it was shown by Akhlil \cite{Khalid2018} that the assumption of the locality of the form is redundant. More recently, Chill, Djida, and the first author revisited domination of semigroups generated by regular forms. The methods used there also show that the sandwiched semigroup must correspond to a local form; see \cite[Section~4.1]{AroraChillDjida2021}. 
Here, we show that the semigroup generated by the operator considered in \cite[Example~4.5]{ArendtWarma2003} is eventually sandwiched between the Dirichlet and Neumann semigroups. Additionally, we prove eventual domination of the resolvents using our results in Section~\ref{domination-resolvents}.

On $L^2(0,1)$ we consider the Dirichlet Laplacian $\Delta^D$ with domain
\[
	\dom{\Delta^D}=\{f\in H^2(0,1) : f(0)=f(1)=0\},
\]
the Neumann Laplacian $\Delta^N$ with domain
\[
	\dom{\Delta^N}=\{f\in H^2(0,1) : f'(0)=f'(1)=0\},
\]
and a Laplace operator $\Delta^{\nl}$ with non-local boundary conditions with domain
\[
	\dom{\Delta^{\nl}}=\{f\in H^2(0,1) : f'(0)=-f'(1)=f(0)+f(1)\}.
\]

It is well-known that both $\Delta^D$ and $\Delta^N$ generate positive analytic $C_0$-semigroups on $L^2(0,1)$. By contrast, while $\Delta^{\nl}$ does generate an analytic semigroup, the generated semigroup is not positive but merely uniformly eventually positive \cite[Theorem~4.2]{DanersGlueck2018b}. Moreover, uniform maximum and anti-maximum principles for the operator $\Delta^{\nl}$ were recently proved in \cite[Proposition~6.2]{AroraGlueck2021b}.

\begin{theorem}
	\label{thm:dtn-resolvent-dominance}
	For the Laplace operators considered above, we have the following domination of the resolvents.
	\begin{enumerate}[\upshape (i)]
		\item For each $0\neq f\in L^2(0,1)$, we have
			\[
				\Res(\lambda,\Delta^{\nl} )\modulus{f}- \modulus{\Res(\lambda, \Delta^D)f}\succeq \one
			\]
			and
			\[
				\Res(\mu, \Delta^{\nl})\modulus{f}+\modulus{\Res(\mu, \Delta^D)f}\preceq -\one
			\]
	 		for all $\lambda$ in an $f$-dependent right neighbourhood of $\spb(\Delta^{\nl})$ and for all $\mu$ in an $f$-dependent left neighbourhood of $\spb(\Delta^{\nl})$.
	 
	 	\item For each $0\neq f\in L^2(0,1)$, we have
			\[
				\Res(\lambda,\Delta^{N} )\modulus{f}- \modulus{\Res(\lambda, \Delta^{\nl})f}\succeq \one
			\]
			and
			\[
				\Res(\mu, \Delta^N)\modulus{f}+\modulus{\Res(\mu, \Delta^{\nl})f}\preceq -\one
			\]
	 		for all $\lambda$ in an $f$-dependent right neighbourhood of $0$ and for all $\mu$ in an $f$-dependent left neighbourhood of $0$.
	\end{enumerate}
\end{theorem}

\begin{proof}
	We simultaneously verify the assumptions of Theorem~\ref{thm:domination-resolvents-individual} for both parts. To begin, note that the domain of each of the three operators lies inside $L^{\infty}(0,1)=L^2(0,1)_{\one}$. Also, each of the domains embeds into $H^1(0,1)\hookrightarrow L^2(0,1)$ and the latter embedding is compact. Therefore, each of the operators has compact resolvent (in particular, each spectral value is a pole of the corresponding resolvent).
	
	Now, it was shown in \cite[Theorem~4.4]{GlueckMugnolo2021} that $\spb(\Delta^D)<\spb(\Delta^{\nl})$ and in \cite[Lemma~6.9]{DanersGlueckKennedy2016b} that $\spb(\Delta^{\nl})<0$. Moreover, using the boundary conditions, one can check that $\spb(\Delta^N)=0$. It follows that $\spb(\Delta^{\nl})\in \rho(\Delta^D)$ and $0\in\rho(\Delta^{\nl})$. In addition, assumption~(b) in Theorem~\ref{thm:domination-resolvents-individual} for the operators $\Delta^N$ and $\Delta^{\nl}$ was verified in \cite[Propositions~6.1(b) and~6.2]{AroraGlueck2021b}, respectively (more precisely, in \cite[Propositions~6.1(b) and~6.2]{AroraGlueck2021b} a property was verified that is, according to \cite[Corollary~3.3]{DanersGlueckKennedy2016b}, equivalent to assumption~(b) in Theorem~\ref{thm:domination-resolvents-individual}).
	
	With the above observations, the assertions readily follow from Theorem~\ref{thm:domination-resolvents-individual}.
\end{proof}

As announced before, we are going to show that the semigroup generated by $\Delta^{\nl}$ is eventually sandwiched between the Dirichlet and the Neumann semigroup. We point out that the eventual domination of the Dirichlet semigroup by the semigroup generated by $\Delta^{\nl}$ is already known from \cite[Theorem~4.4]{GlueckMugnolo2021}.

\begin{theorem}
	There exists a time $t_0\geq 0$ such that $e^{t\Delta^D}\leq e^{t\Delta^{\nl}}\leq e^{t\Delta^N}$ for all $t\geq t_0$.
\end{theorem}

\begin{proof}
	First of all, it was already proved in \cite[Theorem~4.4]{GlueckMugnolo2021} that there exists a time $t_1\geq 0$ such that
	\[
		e^{t\Delta^D}\leq e^{t\Delta^{\nl}}
	\]
	for all $t\geq t_1$.
	
	For the second estimate, we verify the assumptions of Theorem~\ref{thm:domination-semigroups-uniform}. Note that both the operators $\Delta^N$ and $\Delta^{\nl}$ are self-adjoint. Moreover, the corresponding semigroups map $L^2(0,1)$ into $H^2(0,1)\subseteq L^{\infty}(0,1)=L^2(0,1)_{\one}$.
	
	Next, note that for the Neumann Laplacian, the spectral bound $\spb(\Delta^N)=0$ is a dominant spectral value and the corresponding eigenspace is spanned by $\one$. Therefore by Remark~\ref{remark:domination-semigroups-uniform}(b), we obtain that $\Delta^N$ satisfies the assumption (b) of Theorem~\ref{thm:domination-semigroups-uniform}. 
	On the other hand, 
	because $\Delta^{\nl}$ generates an analytic semigroup, the growth bound of this semigroup satisfies $\gbd(\Delta^{\nl})=\spb(\Delta^{\nl})$. Moreover, it was shown in \cite[Lemma~6.9]{DanersGlueckKennedy2016b} that $\spb(\Delta^{\nl})<0$.
	In particular, the semigroup generated by $\Delta^{\nl}$ converges to $0$ with respect to the operator norm as $t \to \infty$. 
	
	We now conclude from Theorem~\ref{thm:domination-semigroups-uniform} that there exists $t_2\geq 0$ such that
	\[
		e^{t\Delta^D}\leq e^{t\Delta^{\nl}} \quad \text{and}\quad e^{s\Delta^{\nl}}\leq e^{s\Delta^N}
	\]
	for all $t\geq t_1$ and for all $s\geq t_2$; the assertion is now immediate.
\end{proof}

\subsection{Bi-Laplace operator with Wentzell boundary conditions}

Let $\Omega\subseteq \bbR^d$ be a bounded domain with Lipschitz boundary, where $\Omega$ is equipped with the Lebesgue measure and its boundary is equipped with the surface measure $S$. Fix functions $\alpha \in L^{\infty}(\Omega,\bbR)$ and $\beta,\gamma_1,\gamma_2 \in L^{\infty}(\partial\Omega,\bbR)$ such that there is a number $\eta > 0$ satisfying $\eta \leq \alpha$ and $\eta\leq \beta$ almost everywhere. Moreover, we assume that $\gamma_1,\gamma_2 \geq 0$.

Our state space is the Hilbert lattice
\[
	H:= L^2(\Omega)\times L^2(\partial\Omega, \beta^{-1}dS)
\]
equipped with the inner product
\[
	\duality{u}{v} = \int_{\Omega} u_1 \overline{v_1}\dx x  + \int_{\partial\Omega} u_2\overline{v_2}\beta^{-1}\dx S 
\]
for $u=(u_1,u_2)$ and $v=(v_1,v_2)$ in $H$, and we consider the operators $-A_i$ associated with the forms
\[
	a_i((u_1,u_2),(v_1,v_2)) := \duality{(\alpha\Delta u_1, \gamma_i u_2)}{(\Delta v_1, v_2)}
\]
with form domains
\[
	\dom{a_i} := \Big\{ (u_1,u_2)\in H: u_1\in \dom{\Delta^N}, u_2=\tr u_1 \Big\}
\]
for $i=1,2$; here $\Delta^N$ denotes the Neumann Laplacian on $\Omega$. It was shown in \cite[Theorem~3.4]{DenkKunzePloss2021} that the operators $A_1$ and  $A_2$ are self-adjoint and generate analytic and contractive $C_0$-semigroups on $H$. In \cite{DenkKunzePloss2021}, it was also shown that the semigroups are not positive but eventually positive if $\gamma_i=0$. Combining the results of the aforementioned reference with our current results, we are able to obtain eventual domination for both semigroups and resolvents:

\begin{theorem}
	Let $\gamma_2=0$ almost everywhere and let $\gamma_1 \geq 0$ be non-zero on a set of non-zero measure. 
	In addition, set $\one:= (\one_{\Omega},\one_{\partial\Omega})$. 
	Then there exists $\tau\geq 0$ such that 
	\[
		e^{tA_2} \modulus{f} - \modulus{e^{tA_1}f}  \succeq  (\one\otimes \one) \modulus{f}
	\]
	for all $t\geq \tau$ and all $f\in E$.
\end{theorem}

\begin{proof}
	Fix $i \in \{1,2\}$. Since the semigroups are analytic, we have
	\[
		e^{tA_i}H\subseteq \cap_{n\geq 0} \dom{A_i^n}
	\]
	for all $t>0$. Moreover, it follows from \cite[Theorem~5.4]{DenkKunzePloss2021} that $\cap_{n\geq 0} \dom{A_i^n}$ embeds into $L^{\infty}(\Omega)\times L^{\infty}(\partial\Omega)= H_{\one}$. By virtue of \cite[Lemma~6.4(i)]{DenkKunzePloss2021}, we have that $\spb(A_2)=0$ and the corresponding eigenspace is spanned by $\one$. Since the operators $A_1$ and $A_2$ are self-adjoint, it follows by Remark~\ref{remark:domination-semigroups-uniform}(b) that assumptions (a) and (b) of Theorem~\ref{thm:domination-semigroups-uniform} are satisfied. In fact, according to \cite[Theorem~6.5(ii)]{DenkKunzePloss2021}, assumption (c) is also satisfied. Thus the assertion follows.
\end{proof}

\begin{theorem}
	Let $\gamma_2=0$ almost everywhere and let $\gamma_1 \geq 0$ be non-zero on a set of non-zero measure. 
	In addition, set $\one:= (\one_{\Omega},\one_{\partial\Omega})$. If $d\leq 5$, 
	then for each $f\in E$, we have
	\[
		\Res(\lambda, A_2)\modulus{f}- \modulus{\Res(\lambda, A_1)f}\succeq \one \quad \text{ and }\quad \Res(\mu, A_2)\modulus{f}+\modulus{\Res(\mu, A_1)f}\preceq -\one
	\]
	 for all $\lambda$ in an $f$-dependent right neighbourhood of $0$ and for all $\mu$ in an $f$-dependent left neighbourhood of $0$.
\end{theorem}

\begin{proof}
	Since $d\leq 5$, the proof of \cite[Theorem~5.4]{DenkKunzePloss2021} implies that
	\[
		\dom{A_i}\subseteq L^{\infty}(\Omega)\times L^{\infty}(\partial\Omega)= H_{\one}\qquad (i=1,2).
	\]
	Moreover, by \cite[Lemma~6.4(i)]{DenkKunzePloss2021}, we have that $\spb(A_2)=0$ with corresponding eigenspace being spanned by $\one$. 
	Since the operator $A_2$ is self-adjoint, it follows that the spectral projection $P$ of $A_2$ associated with $0$ is a multiple of $\one \otimes \one$, and thus satisfies $Pf\succeq \one$ for all $0\lneq f\in E$. Lastly, we know from \cite[Lemma~6.4(ii)]{DenkKunzePloss2021} that $0\in \rho(A_1)$. In particular, all assumptions of Theorem~\ref{thm:domination-resolvents-individual} are satisfied which yields both the assertions.
\end{proof}

\subsection{Differential operators of odd order}
	\label{sec:application-odd-order}

We conclude this section by considering differential equations of odd order on an interval; these were recently considered in \cite[Section~6.4]{AroraGlueck2021b}. It turns out, a little unexpectedly, that these operators satisfy a uniform maximum and anti-maximum principle. Furthermore, even though the first order operator generates a positive semigroup,
if the order is strictly larger than one, then the corresponding semigroup is not even individually eventually positive. Here, we show that there is no eventual domination of resolvents.

Fix integers $m,\ell\in\bbN_0$.
On the space $L^2(0,1)$, consider the operators
\begin{align*}
	\dom{A_k} &:= \left\{f \in H^{2k+1}(0,1): f^{(j)}(0)=f^{(j)}(1) \text{ for all } j=0,1,\ldots, 2k\right\}\\
	          A_k &:= f^{(2k+1)}
\end{align*}
for $k=m,\ell$.

Let $k\in \{m,\ell\}$.
The number $0$ is spectral value of $A_k$ and a simple pole of the resolvent $\Res(\argument,A_k)$ \cite[Propositions~6.8 and~3.1(b)]{AroraGlueck2021b}. 
Also, according to  
\cite[Theorem~6.9]{AroraGlueck2021b}, we have the estimate $\Res(\lambda,A_k)\succeq \one\otimes \one$ for all $\lambda$ in a right neighbourhood of $0$. Now if $m\neq \ell$, then $A_m\neq A_{\ell}$. Thus by Theorem~\ref{thm:domination:converse}, we have proved the following:

\begin{theorem}
	If for each $0\leq f\in L^2(0,1)$, we have $\Res(\lambda, A_m)f\leq \Res(\lambda,A_{\ell})f$ for all $\lambda$ in an $f$-dependent right neighbourhood of $0$, then $m=\ell$.
\end{theorem}

\section{Some remarks on eventual positivity of Cesàro means}
	\label{cesaro-means}

We close the paper with a few remarks on the eventual positivity of Cesàro means of $C_0$-semigroups. Since Cesàro means behave in several respects similarly to resolvents, it is reasonable to expect comparable behaviour concerning eventual positivity; the following results give a few first indications that this expectation is justified.

Let $E$ be a complex Banach lattice and let $(e^{tA})_{t\geq 0}$ be a real $C_0$-semigroup on $E$. For a fixed quasi-interior point $u\in E$, we say that the Cesàro means of $(e^{tA})_{t\geq 0}$ are \emph{individually eventually strongly positive with respect to $u$} if for each $0\lneq f\in E$, there exists a time $t_0\geq 0$ such that $C(r)f\succeq u$ for all $t\geq 0$; 
here the operators $C(r) \in \calL(E)$ are defined by
\[
	C(r)f:=\frac1r \int_0^r e^{sA} f\, ds \qquad \text{for all } f\in E
\]
and are called the \emph{Cesàro means of $(e^{tA})_{t\geq 0}$}. 
Some properties of the Cesàro means are given in \cite[Lemma~V.4.2]{EngelNagel2000}. Our first result here is the following criterion for the Cesàro means to be individually eventually positive.

\begin{proposition}
	\label{prop:ev-pos-cesaro-means}
	Let $(e^{tA})_{t\geq 0}$ be a real $C_0$-semigroup on a complex Banach lattice $E$ and let $u \ge 0$ be a quasi-interior point of $E$. Assume that $\spec(A)$ is non-empty and consider the following assertions:
	\begin{enumerate}[\upshape (i)]
		\item The Cesàro means are individually eventually strongly positive with respect to $u$.
		
		\item The rescaled semigroup $(e^{t(A-\spb(A))})_{t\geq 0}$ is mean ergodic, satisfies 
		\[
			\lim_{t\to \infty} \frac{1}{t}\norm{e^{t(A-\spb(A))}} = 0,
		\]
		and the mean ergodic projection $P$ satisfies $Pf\succeq u$ for all $0\lneq f\in E$.
		
		\item The rescaled semigroup $(e^{t(A-\spb(A))})_{t\geq 0}$ is Abel ergodic, bounded 
		and the Abel ergodic projection $Q$ satisfies $Qf\succeq u$ whenever $0\lneq f\in E$.
	\end{enumerate}
	We always have that {\upshape(iii)} implies {\upshape(ii)}. If, in addition, $\dom{A}\subseteq E_u$, then {\upshape(ii)} implies {\upshape(i)}.
\end{proposition}

Before giving the proof, we briefly recall the notions used above. A $C_0$-semigroup $(e^{tA})_{t\geq 0}$ on a Banach space $E$ is called \emph{mean ergodic} if its Cesàro means $C(r)$ converge strongly as $r\to \infty$. If this is the case and, in addition $\norm{e^{t(A-\spb(A))}}/t \to 0$ as $t \to \infty$, then the limit operator 
\[
E\ni f\mapsto Pf:= \lim_{r\to\infty} C(r)f
\]
is a projection and is called the corresponding \emph{mean ergodic projection}. A brief overview of mean ergodicity of semigroups is given in \cite[Section~V.4]{EngelNagel2000}. Furthermore, the semigroup $(e^{tA})_{t\geq 0}$  is called \emph{Abel ergodic} if there exists $\lambda_0>0$ such that $(0,\lambda_0)\subseteq \resSet(A)$ and
\[
	\sup_{\lambda\in (0,\lambda_0)} \norm{\lambda\Res(\lambda,A)}<\infty
\]
and the limit $\lim_{\lambda\downarrow 0} \lambda\Res(\lambda,A)$ exists in the strong operator topology. In this case, the operator
\[
	E\ni f\mapsto Qf:=  \lim_{\lambda\downarrow 0} \lambda\Res(\lambda,A)f
\]
is also a projection and is called the corresponding \emph{Abel ergodic projection}. We refer to \cite[Section~4.3]{ArendtBattyHieberNeubrander2011} for details about Abel ergodicity.

\begin{proof}[Proof of Proposition~\ref{prop:ev-pos-cesaro-means}]
	We assume without loss of generality that $\spb(A)=0$. 
	
	``(iii) $\Rightarrow$ (ii)''
	If the semigroup generated by $A-\spb(A)$ is Abel ergodic and bounded, then it is mean ergodic and the corresponding ergodic projections coincide by \cite[Proposition~4.3.4(b)]{ArendtBattyHieberNeubrander2011}.
		 
	We now assume that $\dom{A}\subseteq E_u$.
	
	``(ii) $\Rightarrow$ (i)'': Let $C(r):=\frac1r\int_0^r e^{sA}\, ds$ denote the Cesàro means of the semigroup. Since the semigroup is mean ergodic, we have $C(r)\to P$ strongly as $r\to \infty$ and $\Ima P=\ker A$. Therefore
	\[
		AC(r) = \frac{A}r\int_0^r e^{sA}\, ds=\frac{e^{rA}-I}{r} \to 0 = AP
	\]
	with respect to the strong operator topology in $E$ as $r\to\infty$. As a result, $C(r)f\to Pf$ in $\dom{A}$ (endowed with the graph norm) as $r\to \infty$ for each $f \in E$. Since $\dom{A}$ embeds continuously into $E_u$ (by the closed graph theorem) the convergence holds in $E_u$ as well. 
	
	As the semigroup is real, so are the Cesàro means $C(r)$ and thus, assertion~(i) follows from Lemma~\ref{lem:convergence-strong-positivity}.
\end{proof}

Under certain technical assumptions, one can give a characterisation for eventual positivity of the Cesàro means which shows that this property is closely related to (anti-)maximum principles.

\begin{theorem}
	\label{thm:ev-pos-cesaro-means}
	Let $(e^{tA})_{t\geq 0}$ be a real $C_0$-semigroup on a complex Banach lattice $E$ and let $0\leq u\in E$ be such that $\dom{A}\subseteq E_u$. Assume that the spectral bound $\spb(A)$ is a pole of the resolvent $\Res(\argument,A)$ and that the rescaled semigroup $(e^{t(A-\spb(A))})_{t\geq 0}$ is bounded. 
	Then the following are equivalent.
	\begin{enumerate}[\upshape (i)]
		\item The Cesàro means are individually eventually strongly positive with respect to $u$.
		
		\item The spectral projection $P$ corresponding to $\spb(A)$ satisfies $Pf\succeq u$ whenever $0\lneq f\in E$.
		
		\item The individual strong maximum principle with respect to $u$ holds, i.e., for every $0 \lneq f \in E$, we have 
			\begin{align*}
				\Res(\lambda,A)f \succeq u
			\end{align*}
			for all $\lambda$ in an $f$-dependent right neighbourhood of $\lambda_0$.
			
		\item The individual strong anti-maximum principle with respect to $u$ holds, i.e., for every $0 \lneq f \in E$, we have 
			\begin{align*}
				\Res(\mu,A)f \preceq -u
			\end{align*}
			for all $\mu$ in an $f$-dependent left neighbourhood of $\lambda_0$.
	\end{enumerate}
\end{theorem}

\begin{proof}
	By replacing $A$ by $A-\spb(A)$, we may assume that $\spb(A)=0$. The equivalence of (ii), (iii), and (iv) under the assumption $\dom{A}\subseteq E_u$ was proved in \cite[Theorem~4.4]{DanersGlueckKennedy2016b}.
	On the other hand, boundedness of the semigroup implies that $0$ is actually a simple pole of the resolvent $\Res(\argument,A)$. We infer that $(e^{tA})_{t\geq 0}$ is uniformly mean ergodic and the corresponding mean ergodic projection coincides with the spectral projection $P$ associated with $0$; see \cite[Theorem~ V.4.10]{EngelNagel2000}. Since $\dom{A}\subseteq E_u$, the implication (ii) $\Rightarrow$ (i) now holds due to Proposition~\ref{prop:ev-pos-cesaro-means}.
	
	Lastly, if (i) holds, then so does (ii) by \cite[Lemma~4.8(ii)]{DanersGlueckKennedy2016b}.
\end{proof}

It has recently been shown by the present authors \cite[Theorem~1.2]{AroraGlueck2022a} that the condition $\dom{A} \subseteq E_u$ is, under mild assumptions on $A$, necessary for assertions~(iii) and~(iv) to hold simultaneously. 

Clearly, the above results can only be considered as a starting point for a more thorough analysis, in particular since the methods used in this section are very closely related to the methods used in earlier articles about eventual positivity. 
In the future, it would, for instance, be desirable to better understand the relation between eventual positivity of the resolvent and the Cesàro means, without imposing the a priori assumption $\dom{A} \subseteq E_u$.
However, this seems to require more sophisticated analysis and probably different techniques.

We end this section with the observation that, by using Theorem~\ref{thm:ev-pos-cesaro-means}, it can be shown for several semigroups considered in \cite{DanersGlueckKennedy2016a,DanersGlueckKennedy2016b, AroraGlueck2021b}, that the Cesàro means are individually eventually positive.

\subsection*{Acknowledgements} 
The authors are grateful to Delio Mugnolo for pointing out the reference \cite{Ouhabaz1996}.

Part of the work on this article was done during a pleasant visit of the first author to the second author at Universität Passau, Germany in 2021.
The first named author was supported by 
Deutscher Aka\-de\-mi\-scher Aus\-tausch\-dienst (Forschungs\-stipendium-Promotion in Deutschland).

\bibliographystyle{plainurl}
\bibliography{literature}

\end{document}